\documentclass[12pt]{amsart}
\usepackage[all]{xy}
\usepackage{graphicx}
\usepackage{amsmath,amsxtra,amssymb,latexsym, amscd, amsthm, amscd}
\usepackage{indentfirst}
\usepackage[mathscr]{eucal}

\newtheorem{thm}{Theorem}[section]
\newtheorem{cor}[thm]{Corollary}
\newtheorem{lem}[thm]{Lemma}
\newtheorem{prop}[thm]{Proposition}
\theoremstyle{definition}

\newtheorem{rem}[thm]{Remark}

\DeclareMathOperator{\lk}{lk}
\DeclareMathOperator{\st}{st}
\DeclareMathOperator{\ch}{char}
\DeclareMathOperator{\core}{core}

\DeclareMathOperator{\girth}{girth}
\DeclareMathOperator{\type}{type}
\DeclareMathOperator{\R}{\mathbb {RP}}
\DeclareMathOperator{\im}{im}

\def\k{{\widetilde\chi}}
\def\h{{\widetilde H}}

\def\zv {\mathbf 0}
\def\e {\mathbf e}

\begin{document}

\title[A Characterization of Triangle-free Gorenstein graphs and Cohen-Macaulayness of second powers of edge ideals]
 {A Characterization of Triangle-free Gorenstein graphs and Cohen-Macaulayness of second powers of edge ideals}

\author{Do Trong Hoang}
\address{Institute of Mathematics, VAST, 18 Hoang Quoc Viet, Hanoi, Viet Nam}
\email{dthoang@math.ac.vn}

\author{ Tran Nam Trung}
\address{Institute of Mathematics, VAST, 18 Hoang Quoc Viet, Hanoi, Viet Nam}
\email{tntrung@math.ac.vn}

\date{}
\dedicatory{}
\commby{}
%-----------------------------------------------------------
\begin{abstract}
We graph-theoretically characterize triangle-free Gorenstein graphs $G$. As an application, we classify when $I(G)^2$ is Cohen-Macaulay.
\end{abstract}
% -----------------------------------------------------------
\maketitle
% -----------------------------------------------------------
\section*{Introduction}

Throughout this paper let $G$ be a simple graph (i.e., a finite, undirected, loopless and without multiple edges) with vertex set $V(G)$ and edge set $E(G)$ and let $k$ be a fixed field. Two vertices $u, v$ of $G$ are {\it adjacent} if $uv$ is an edge of $G$. An {\it independent set}  in $G$ is a set of vertices no two of which are adjacent to each other. An independent set in $G$ is  maximal (with respect to set inclusion) if the set cannot be extended to a larger independent set. The independence number of $G$, denoted by $\alpha(G)$, is the cardinality of the largest independent set in $G$. The girth of a graph $G$, denoted by $\girth(G)$, is the length of any shortest cycle in $G$ or in the case $G$ is a forest we consider the girth to be infinite.  If $G$ has girth at least $4$, then $G$ is {\it triangle-free}.

In a $1970$, Plummer \cite{P1} introduced the notion of considering graphs in which each maximal independent set has the same size $\alpha(G)$; he called a graph having this property a {\it well-covered} graph. In general, to compute $\alpha(G)$ of a graph $G$ is an NP-complete problem (see \cite{K}), but it  is polynomial for well-covered graphs. Characterize well-covered graphs is a difficult problem and the work on well-covered graphs that has appeared in the literature has focused on certain subclasses of well-covered graphs (see \cite{P2} for the survey).

Recently, well-covered graphs with the Cohen-Macaulay property have been studied from an algebraic point of view. Let $R=k[x_1,\ldots,x_n]$ be a polynomial ring of $n$ variables over the field $k$. Let $G$ be a simple graph on the vertex set $\{x_1, \ldots , x_n\}$. We associate to the graph $G$ a quadratic squarefree monomial ideal $$I(G) = (x_ix_j  \mid x_ix_j \in E(G)) \subseteq R,$$ which is called the {\it edge ideal} of $G$. We say that $G$ is {\it Cohen-Macaulay} (resp. {\it Gorenstein}) if $I(G)$ is Cohen-Macaulay (resp. Gorenstein). Since $G$ is well-covered whenever it is Cohen-Macaulay (see e.g., \cite[Proposition $6.1.21$]{Vi2}), it is a wide open problem to characterize graph-theoretically the Cohen-Macaulay graphs. Unfortunately, the Cohen-Macaulayness of graphs depends on the characteristic of the base field $k$ (see, e.g., \cite[Exercise $5.3.31$]{Vi2}), so the characterization now has focused on certain families of Cohen-Macaulay graphs,  such as trees, chordal graphs, bipartite graphs, graphs of girth at least $5$, circulant graphs and so on (see \cite{Vi1}, \cite{HHZ}, \cite{HH1}, \cite{HMT2}, \cite{VVW}).

If we focus on Gorenstein graphs $G$ without isolated vertices, then not only is  $G$ well-covered, but is $G\setminus x$ also well-covered, with $\alpha(G)=\alpha(G\setminus x)$ for any vertex $x$. Such graphs form the so-called class $W_2$ (see \cite{Sp2}). Thus in order to characterize  Gorenstein graphs we should classify the class $W_2$, but this problem is also difficult (see \cite{P2}). Note that the Gorenstein property of graphs is also dependent on the characteristic of the base field $k$ (see Proposition \ref{GK}), so we cannot  graph-theoretically characterize all Gorenstein graphs. In this paper, we are interested in triangle-free Gorenstein graphs. The answer is amazing: a triangle-free graph $G$ is Gorenstein if and only if every non-trivial connected component of $G$ belongs to $W_2$.

Let $\Delta$ be a simplicial complex over the vertex set $\{x_1,\ldots,x_n\}$. A classical result in \cite{CN} implies that $I_{\Delta}^m$ is Cohen-Macaulay ({\it over} $k$) for all $m\geqslant 1$ if and only if $\Delta$ is a complete intersection; where $I_{\Delta}$ is the Stanley-Reisner ideal of $\Delta$. Recently, Terai and Trung \cite{TT} showed that $\Delta$ is a complete intersection whenever $I_{\Delta}^m$ is Cohen-Macaulay for some $m\geqslant 3$. By contrast with this situation we have not known a characterization of $\Delta$ for which $I_{\Delta}^2$ is Cohen-Macaulay yet (see \cite{HMT1}, \cite{MT}, \cite{RTY}, \cite{TT}). On the other hand, Vasconcelos (see \cite[Conjecture (B)]{Va}) suggests that $\Delta$ must be Gorenstein if $I_{\Delta}^2$ is Cohen-Macaulay.

For edge ideals, we will prove that $I(G)^2$ is Cohen-Macaulay if and only if $G$ is triangle-free Gorenstein . The main result of the paper is the following theorem.
\smallskip

\noindent {\bf Theorem $\ref{mth}.$}{\it $\ $ Let $G$ be a simple graph without isolated vertices. Then, the following conditions are equivalent:
\begin{enumerate}
\item $G$ is a triangle-free Gorenstein graph;
\item $G$ is a triangle-free member of $W_2$;
\item $I(G)^2$ is Cohen-Macaulay.
\end{enumerate}
}
\smallskip
The paper is organized as follows. In Section $1$, we recall some basic notations, terminology of the simplicial complexes. In Section $2$, we characterize triangle-free Gorenstein graphs. In the last section, we prove the main result.

\section{Preliminaries}

Let $\Delta$ be a simplicial complex on the vertex set $V(\Delta)=[n]=\{1,2,\ldots,n\}$. Given any field $k$, we attach to $\Delta$ the {\it Stanley-Reisner} ideal $I_{\Delta}$ of $\Delta$ to be the squarefree monomial ideal
$$I_{\Delta} = (x_{j_1} \cdots x_{j_i} \mid j_1  <\cdots < j_i \ \text{ and } \{j_1,\ldots,j_i\} \notin \Delta) \ \text{ in } R = k[x_1,\ldots,x_n]$$
and the {\it Stanley-Reisner} ring of $\Delta$ to be the quotient ring $k[\Delta] = R/I_{\Delta}$.  This provides a bridge between combinatorics and commutative algebra (see \cite{S}). Then, we say that $\Delta$ is Cohen-Macaulay (resp. Gorenstein) over $k$ if $k[\Delta]$ has the same property.

The dimension of a face $F \in\Delta$ is given by $\dim F = |F|-1$; the dimension of $\Delta$, denoted $\dim \Delta$, is the maximum dimension of all its faces. Note that $\dim k[\Delta] = \dim \Delta + 1$. The link of $F$ inside $\Delta$ is the subcomplex of $\Delta$:
$$\lk_{\Delta}  F = \{H\in \Delta \mid H\cup F\in \Delta \ \text{ and } H\cap F=\emptyset\}.$$

The most widely used criterion for determining when a simplicial complex is Cohen-Macaulay is due to Reisner, which says that links have only top homology (see \cite[Corollary $4.2$, Page $60$]{S}).

\begin{lem} \label{Reisner} $\Delta$ is Cohen-Macaulay over $k$ if and only if for all $F\in \Delta$ and all $i < \dim(\lk_{\Delta} F)$, we have $\widetilde{H}_i(\lk_{\Delta} F; k) = \zv$.
\end{lem}

Let $f_i$ be the number of faces of $\Delta$ of dimension $i$ for each $i$, and let $d :=\dim \Delta+1$. Then the $f$-vector of $\Delta$ is $f(\Delta) = (f_{-1},f_0, f_1, \ldots, f_{d-1})$; and this vector is related to  the dimensions of  $\h_i(\Delta; k)$ via the {\it reduced Euler characteristic} of $\Delta$:
$$\k(\Delta) :=  \sum_{i=-1}^{d-1} (-1)^i\dim_k\h_{i}(\Delta;k) =\sum_{i=-1}^{d-1} (-1)^i f_i = \sum_{F\in \Delta} (-1)^{|F|-1}.$$
We call $\Delta$ a pure simplicial complex if all its facets have the same dimension;  and $\Delta$ is an Eulerian complex if $\Delta$ is pure and $\k(\lk_{\Delta} F) = (-1)^{\dim \lk_{\Delta} F}$ for all $F\in \Delta$. In this case, note that $\k(\lk_{\Delta} F) = \k(\mathbb S^{d-1-|F|})$ (see \cite[Proposition $3.3$]{S1}). 

The restriction of $\Delta$ to a subset $S$ of $V(\Delta)$ is $\Delta_S :=\{F\in\Delta \mid F \subseteq S\}$. The star of a vertex $v$ in $\Delta$ is $\st_{\Delta}(v) :=\{F\in\Delta \mid F\cup \{v\}\in\Delta\}$. Let $\core(V(\Delta)) := \{x\in V(\Delta) \mid \st_{\Delta}(x) \ne \Delta\}$, then the core of $\Delta$ is $\core(\Delta):= \Delta_{\core(V(\Delta))}$. If $\Delta =\st_{\Delta}(v)$ for some vertex $v$, then $\Delta$ is a cone over $v$. Thus $\Delta=\core(\Delta)$ means $\Delta$ is not a cone. Let $\Delta$ and $\Gamma$ be simplicial complexes with disjoint vertex sets $V(\Delta)$ and $V(\Gamma)$, respectively. Define the join $\Delta * \Gamma$ to be the simplicial complex on the vertex set $V(\Delta) \cup V(\Gamma)$ with faces $F\cup H$, where $F\in\Delta$ and $H\in\Gamma$. It follows that $\Delta = \core(\Delta) *  \left< V(\Delta) \setminus \core(V(\Delta)) \right>$, where for a finite set $P$, we denote $\left<P\right>$ to be the simplex over the set $P$.

We then have a criterion for determining when Cohen-Macaulay complexes $\Delta$ are Gorenstein due to Stanley (see \cite[Theorem $5.1$, Page $65$]{S}).

\begin{lem} \label{Stanley} $\Delta$ is Gorenstein if and only if $\core(\Delta)$ is an Eulerian complex which is Cohen-Macaulay.
\end{lem}

For a subset $S$  of the vertex set of $\Delta$, let $\Delta\setminus S:=\{F\in\Delta\mid F\subseteq V(\Delta)\setminus S\}$. If $S =\{x\}$, we write $\Delta\setminus x$ to mean $\Delta\setminus \{x\}$. Clearly, $\Delta\setminus x = \{F\in\Delta \ | \ x\notin F\}$.

$\Delta$ is called doubly Cohen-Macaulay (see \cite{B}) if $\Delta$ is Cohen-Macaulay, and for every vertex $x$ of $\Delta$ the subcomplex $\Delta\setminus x$ is Cohen-Macaulay of the same dimension as $\Delta$. If $\Delta$ is Cohen-Macaulay, then $k[\Delta]$ has a minimal free resolution of the form:
$$\zv \longrightarrow R^{\beta_{n-d}(k[\Delta])} \longrightarrow\cdots  \longrightarrow R^{\beta_1(k[\Delta])}\longrightarrow R^{\beta_0(k[\Delta])}\longrightarrow k[\Delta] \longrightarrow \zv$$
where $d=\dim k[\Delta]$ and $\beta_i(k[\Delta])$ is the $i$-th Betti number of $k[\Delta]$ for each $i$. The number $\beta_{n-d}(k[\Delta])$ is called the {\it type} of $k[\Delta]$ and is denoted by $\type(k[\Delta])$. We then have $\Delta$ is Gorenstein if and only if $\type(k[\Delta])=1$ (see \cite[Theorem $12.5$, Page $50$]{S}) and $\Delta$ is doubly Cohen-Macaulay if and only if $\type(k[\Delta]) = (-1)^{d-1} \k(\Delta)$ (see \cite{B}). If $\Delta$ is Gorenstein with $\Delta = \core(\Delta)$, then $\k(\Delta) = (-1)^{d-1}$ by Lemma $\ref{Stanley}$. Hence,

\begin{lem}\label{G2CM} If $\Delta$ is a Gorenstein simplicial complex with $\Delta=\core(\Delta)$, then $\Delta$ is doubly Cohen-Macaulay.
\end{lem}

The following lemma is the key to investigate Cohen-Macaulay simplicial complexes in this paper.

\begin{lem}\label{double} Let $\Delta$ be a Gorenstein simplicial complex with $\Delta = \core(\Delta)$. If $S$ is a subset of $V(\Delta)$ such that $\Delta_S$ is a cone, then we have $\h_i(\Delta \setminus S,k) = \zv$ for all $i$.
\end{lem}
\begin{proof}  We prove the lemma by induction on $|S|$. If $|S|=1$, then $S=\{v\}$ for some vertex $v$ of $\Delta$. Let $d:=\dim(\Delta)+1$. By Lemma $\ref{G2CM}$ we have $\Delta$ is doubly Cohen-Macaulay, so $\Delta\setminus v$ is Cohen-Macaulay of dimension $d-1$, and so  by Reisner's criterion we have $\h_i(\Delta\setminus v;k) = \zv$ for all $i < d-1$. Hence it remains to prove, in this case, that  $\h_{d-1}(\Delta\setminus v;k) = 0$. By \cite[Lemma $2.1$]{Hi} we have the following exact sequence:
$$\zv \longrightarrow \h_{d-2}(\lk_{\Delta}(v);k)   \longrightarrow \h_{d-1}(\Delta;k) \longrightarrow   \h_{d-1}(\Delta\setminus v;k)   \longrightarrow \zv.$$
This yields $\dim_k \h_{d-1}(\Delta\setminus v;k) = \dim_k \h_{d-1}(\Delta;k) - \dim_k \h_{d-2}(\lk_{\Delta}(v);k)$.
By Lemma $\ref{Stanley}$ we obtain $\dim_k \h_{d-1}(\Delta;k) = \dim_k \h_{d-2}(\lk_{\Delta}(v);k)=1$. Thus,
$\dim_k \h_{d-1}(\Delta\setminus v;k) =0$, and thus $\h_{d-1}(\Delta\setminus v;k) =\zv$, as claimed.

If $|S|\geqslant 2$. Assume that for all Gorenstein simplicial complexes $\Gamma$ with $\Gamma=\core(\Gamma)$, and for all subsets $T$ of $V(\Gamma)$ such that $|T| < |S|$ and $\Gamma_T$ is a cone; then we have
$$\h_i(\Gamma \setminus T,k) = \zv, \ \text{ for all }\ i.$$

Assume that $\Delta_S$ is a cone with vertex $v$. Take any $x\in S\setminus\{v\}$ and then let $T := S\setminus\{x\}$ and $\Lambda :=\Delta\setminus T$. Since $\Delta_T$ is also a cone with vertex $v$ and $|T| = |S|-1$, by the induction hypothesis we have $\h_i(\Lambda;k) =\h_i(\Delta\setminus T;k) = \zv$. Let $T' := T \cap V(\lk_{\Delta}(x))$, so that $\lk_{\Lambda}(x) =  \lk_{\Delta}(x)\setminus T'$. From Lemma $\ref{Stanley}$ we imply that $\lk_{\Delta}(x)$ is Gorenstein with $\lk_{\Delta}(x) =\core(\lk_{\Delta}(x))$. Note that $\lk_{\Delta}(x)_{T'}$ is a cone over $v$ and $|T'| \leqslant |T| < |S|$, so $\h_i(\lk_{\Lambda}(x); k) = \h_i(\lk_{\Delta}(x)\setminus T'; k) =\zv$ by the induction hypothesis.

Finally, \cite[Lemma $2.1$]{Hi} gives rise to the following exact sequence:
$$\h_i(\Lambda;k)  \longrightarrow \h_i(\Lambda\setminus x;k)  \longrightarrow  \h_i(\lk_{\Lambda}(x);k).$$
Together with the facts $\h_i(\lk_{\Lambda}(x); k)=\zv$ and $\h_i(\Lambda;k)=\zv$ we obtain $\h_i(\Lambda\setminus x;k)=\zv$. As $\Lambda\setminus x = \Delta\setminus S$, which implies $\h_i(\Delta\setminus S;k) = \zv$, as required.
\end{proof}

The lemma $\ref{double}$ has an interesting consequence.

\begin{cor} \label{C1} Let $\Delta$ be a Gorenstein simplicial complex. Then the simplicial complex $\Delta \setminus F$ is Cohen-Macaulay for every face $F$ of $\Delta$.
\end{cor}
\begin{proof} Assume that $\Delta = \core(\Delta) * \left<P\right>$ where $P = V(\Delta)\setminus \core(V(\Delta))$. Let $F_1 := F \cap V(\core(\Delta))$ and $F_2 := P\setminus F$. Then we have $\Delta\setminus F = (\core(\Delta) \setminus F_1) * \left<F_2\right>$, hence $\Delta\setminus F$ is Cohen-Macaulay if so is $\core(\Delta)\setminus F_1$. Therefore, we may assume that $\Delta = \core(\Delta)$.

We now claim that for all faces $S$ of $\Delta\setminus F$ and for all integers $i < \dim(\Delta\setminus F) - |S|$, then $\h_i(\lk_{\Delta\setminus F}(S);k) = \zv$. Indeed, by Lemma $\ref{Stanley}$ we deduce that $\lk_{\Delta}(S)$ is Gorenstein with $\core(\lk_{\Delta}(S)) = \lk_{\Delta}(S)$. Next, we note that $\lk_{\Delta\setminus F}(S) = \lk_{\Delta}(S) \setminus F'$ where $F' = F\cap V(\lk_{\Delta}(S))$. Now if $F' = \emptyset$, then $\lk_{\Delta\setminus F}(S) = \lk_{\Delta}(S)$. Since $\dim(\Delta\setminus F) \leqslant \dim(\Delta)$, the claim follows from Reisner's criterion. If $F'\ne\emptyset$, then $\lk_{\Delta}(S)_{F'}$ is a cone because $F'$ is also a face of  $\lk_{\Delta}(S)$. By Lemma $\ref{double}$ we have
$$\h_{i}(\lk_{\Delta\setminus F}(S);k) = \h_i(\lk_{\Delta F}(S) \setminus F';k) =\zv,$$ 
as claimed.

Together with Reisner's criterion, this claim yields  $\Delta\setminus F$ is Cohen-Macaulay.
\end{proof}

\section{A Characterization of Triangle-free Gorenstein Graphs}

In this section we will graph-theoretically characterize triangle-free Gorenstein graphs. For a graph $G$, let $\Delta(G)$ be the set of all independent sets of $G$. Then, $\Delta(G)$ is a simplicial complex and it is the so-called independence complex of $G$. We can see that $I(G) = I_{\Delta(G)}$ and $\dim(\Delta(G)) = \alpha(G)-1$, where $\alpha(G)$ denotes the independence number of $G$. Notice that $\Delta(G) = \core(\Delta(G))$ if and only if $G$ has no isolated vertices; and $G$ is well-covered if and only if $\Delta(G)$ is pure.

We now start by showing that the Gorenstein property of graphs in general depends on the characteristic of $k$ . Gr\"{a}be\cite{GR} showed that the Gorenstein property of a certain triangulation of $\R^3$ depends on the characteristic of the base field $k$. Actually, any triangulation of $\R^3$ is Gorenstein if and only if $\ch(k)\ne 2$.  Using this fact, we can prove the following fact.

\begin{prop}\label{GK} The Gorenstein property of graphs depends on the characteristic of the base field.
\end{prop}
\begin{proof}
In order to get such a graph let us take a minimal triangulation $\Delta$ of the projective space $\R^3$ with the $f$-vector $f(\Delta) = (1,11, 51, 80, 40)$ constructed by Walkup \cite{W}. Let $\Lambda$ be the barycentric subdivision of $\Delta$, i.e., $\Lambda$ is the simplicial complex on ground set $\Delta \setminus \{\emptyset\}$ whose simplices are flags $F_0 \subsetneq F_1 \subsetneq \cdots \subsetneq F_i$ of  elements $F_j \in  \Delta \setminus \{\emptyset\}, j \leqslant i$. Clearly, the minimal non-faces of $\Lambda$ are subsets of $\Delta$ with exactly two non-comparable elements, therefore $I_{\Lambda}$ is an edge ideal of a graph, say $G$. Since $\Lambda$ is a triangulation of $\R^3$, $G$ is Gorenstein if and only if $\ch(k)\ne 2$. Notice that $G$ has $182$ vertices since $f(\Delta) = (1, 11, 51, 80, 40)$.
\end{proof}

The neighborhood of a vertex $x$ of $G$ is the set $N_G(x) := \{y\in V(G) \mid xy \in E(G)\}$. For an independent set $S$ of $G$ we denote the neighborhood of $S$ by $N_G(S) := \{x \in V(G)\setminus S \mid N_G(x) \cap S \ne\emptyset\}$ and the localization of $G$ with respect to $S$ by $G_S := G\setminus (S\cup N_G(S))$; so that $\Delta(G_S) = \lk_{\Delta(G)}(S)$; where for a subset $U$ of $V(G)$ we denote $G\setminus U$ to be the induced subgraph of $G$ on the vertex set $V(G)\setminus U$.

Note that $G_S$ is Cohen-Macaulay (resp. Gorenstein) if so is $G$ by Lemma $\ref{Reisner}$ (resp. Lemma $\ref{Stanley}$).

\begin{lem}  \cite[Lemma $1$]{FHN}\label{LC} If $G$ is a well-covered graph and $S$ is a set of independent vertices of $G$, then $G_S$ is well-covered. Moreover, $\alpha(G_S) =\alpha(G)-|S|$.
\end{lem}

The same holds for graphs in $W_2$. Recall that a graph $G$ is a member of $W_2$ if $G$ is well-covered and $G\setminus x$ is well-covered with $\alpha(G)=\alpha(G\setminus x)$ for all $x\in V(G)$.

\begin{lem} \label{WLC} If a graph $G$ is in $W_2$ and $S$ is an independent set of $G$, then $G_S$ is in $W_2$ whenever  $|S| < \alpha(G)$.
\end{lem}
\begin{proof} The proof follows immediately by induction on $|S|$ and using \cite[Theorem $5$]{Pi}.
\end{proof}

We say that $G$ is locally Cohen-Macaulay (resp. Gorenstein) if $G_x$ is Cohen-Macaulay (resp. Gorenstein) for all vertices $x$.

\begin{lem} \label{ML} Let $G$ be a locally Gorenstein graph in $W_2$ and let $S$ be a nonempty independent set of $G$. Then we have $G_S$ is Gorenstein and $\Delta(G_S)$ is Eulerian with $\dim(\Delta(G_S))=\dim (\Delta(G))-|S|$.
\end{lem}
\begin{proof} We have $G_S$ is Gorenstein since $G$ is locally Gorenstein and $S$ is nonempty. We now claim that $G_S$ has no isolated vertices. Indeed, if $|S| <\alpha(G)$ then $G_S$ is in $W_2$, so by definition of the class $W_2$ we imply that $G_S$ has no isolated vertices. If $|S| = \alpha(G)$, then $S$ is a maximal independent set of $G$, and then $G_S$ is the empty graph. Therefore, $G_S$ also has no isolated vertices, as claimed. From this claim we have $\Delta(G_S)=\core(\Delta(G_S))$, so $\Delta(G_S)$ is Eulerian by Lemma $\ref{Stanley}$. Finally, by Lemma $\ref{LC}$ we have $\alpha(G_S)=\alpha(G)-|S|$, hence $\dim(\Delta(G_S)) = \alpha(G_S)-1=\alpha(G)-1 - |S| =\dim(\Delta(G))-|S|$, as required.
\end{proof}

We next show that the class $W_2$ indeed contains all Gorenstein graphs without isolated vertices.

\begin{lem} \label{GW} If $G$ is a Gorenstein graph without isolated vertices, then $G$ is in $W_2$.
\end{lem}
\begin{proof} Since $G$ is Gorenstein, it is well-covered. So it remains to prove that for any vertex $x$, we have $G\setminus x$ is well-covered and $\alpha(G)=\alpha(G\setminus x)$. As $G$ has no isolated vertices, $\core(\Delta(G)) = \Delta(G)$. Hence, by Lemma $\ref{G2CM}$ we have $\Delta(G)$ is doubly Cohen-Macaulay. It follows that $\Delta(G\setminus x) = \Delta(G)\setminus x$ is Cohen-Macaulay and $\dim(\Delta(G\setminus x)) = \dim(\Delta(G))$. In other words, $G\setminus x$ is Cohen-Macaulay with $\alpha(G\setminus x)=\alpha(G)$. Thus, $G$ is in $W_2$, as required.
\end{proof}

In general, members of $W_2$ need not be Gorenstein and an example is the $3$-cycle, but any triangle-free member of $W_2$ is Gorenstein. In order to prove this, we first show that the independence complex of a triangle-free member of $W_2$ is Eulerian.

\begin{lem} \label{EulerCh} If $G$ be a triangle-free graph in $W_2$, then $\Delta(G)$ is Eulerian.
\end{lem}
\begin{proof} We prove the lemma by induction on $\alpha(G)$. If $\alpha(G)=1$, then $G$ is complete. Since $G$ is a triangle-free graph without isolated vertices, $G$ must be an edge. It follows that $\Delta(G)$ is Eulerian.

Assume that $\alpha(G)\geqslant 2$. Let $\Delta := \Delta(G)$ and $d :=\dim \Delta(G)+1$. For each vertex $x$ of $G$, by Lemma $\ref{WLC}$ we have $G_x$ is in $W_2$. Since $\lk_{\Delta}(x) = \Delta(G_x)$ and $\alpha(G_x)=\alpha(G)-1$, by the induction hypothesis we have $\lk_{\Delta}(x)$ is Eulerian. So in order to prove $\Delta$ is Eulerian it remains to show that $\k(\Delta)=(-1)^{d-1}$.

Let $a$ be a vertex of $G$, and let $A := N_G(a)$; so that $A$ is a nonempty independent set of $G$ because $G$ is triangle-free.

Let
$$\Gamma := \{F\in \Delta \mid F\cap A \ne\emptyset\},$$
so that $\Delta$ can be partitioned into $\Delta= \st_{\Delta}(a) \cup  \Gamma$.  Note that $\k(\st_{\Delta}(a))=0$ as $\st_{\Delta}(a)$ is a cone over $a$. Thus, 
\begin{align*}
\k(\Delta) = \sum_{F\in\Delta} (-1)^{|F|-1} &= \sum_{F\in\st_{\Delta}(a)} (-1)^{|F|-1}+\sum_{F\in\Gamma} (-1)^{|F|-1}\\
&= \k(\st_{\Delta}(a)) + \sum_{F\in\Gamma} (-1)^{|F|-1} = \sum_{F\in\Gamma} (-1)^{|F|-1}.
\end{align*}
Let $\Lambda := \left<A\right>$ be the simplex over the vertices in $A$, and let $\Omega := \Lambda \setminus \{\emptyset\}$. For each $S\in\Omega$, we define
$$g(S):=\sum_{F\in\Gamma, S \subseteq F} (-1)^{|F|-1} \ \text{ and } \tau(S):=\sum_{F\in\Gamma, F \cap A = S} (-1)^{|F|-1}.$$
Then, \begin{equation}\label{F1} \k(\Delta) = \sum_{S\in\Omega}\tau(S) \ \text{ and } \ g(S) = \sum_{F\in \Omega, S\subseteq F} \tau(F).\end{equation}
For every $S\in \Omega$, since $S$ is a nonempty face of $\Delta$, we have $\Delta(G_S)$ is Eulerian with $\dim (\Delta(G_S)) =\dim (\Delta(G))-|S| = d-1 -|S|$. It follows that $\k(\Delta(G_S)) = (-1)^{d -1-|S|}$. Therefore,
\begin{align*}g(S) &= \sum_{F\in\Gamma, S \subseteq F} (-1)^{|F|-1} =\sum_{F\in\Delta, S \subseteq F} (-1)^{|F|-1}= \sum_{F\in\Delta(G_S)} (-1)^{|F|+|S|-1} \\
&= (-1)^{|S|} \sum_{F\in\Delta(G_S)} (-1)^{|F|-1} = (-1)^{|S|} \k(G_S) = (-1)^{|S|} (-1)^{d-1-|S|} =  (-1)^{d-1}.
\end{align*}
We consider $\Omega$  as a poset with the partial order $\leqslant$ being inclusion. Then, by $(\ref{F1})$ we have
$$g(S) = \sum_{F\in \Omega, F\geqslant S} \tau(F).$$
Let $\mu$ be the Mobius function of the poset $\Omega$. By the Mobius inversion formula (see \cite[Proposition $3.7.2$]{S1}) we have
$$\tau(S) = \sum_{F\in \Omega, F\geqslant S} \mu(S,F)g(F) = \sum_{F\in \Omega, F\geqslant S} \mu(S,F) (-1)^{d-1} = (-1)^{d-1} \sum_{F\in \Omega, F\geqslant S} \mu(S,F).$$
Observe that for $S,F\in\Omega$ with $S\subseteq F$, we have $T\in \Omega$ and $S\leqslant T\leqslant F$ if and only if $S \subseteq T \subseteq F$, so $\mu(S,F) = (-1)^{|F| - |S|}$. Hence,
\begin{align*}\tau(S) &= (-1)^{d-1} \sum_{F\in \Omega, F\geqslant S} (-1)^{|F|-|S|} = (-1)^{d-1} \sum_{F\in \Lambda, S \subseteq F} (-1)^{|F|-|S|}\\
&=-(-1)^{d-1} \sum_{F\in \lk_{\Lambda}(S)} (-1)^{|F|-1} = -(-1)^{d-1}\k(\lk_{\Lambda}(S)).
\end{align*}
Thus, by $(\ref{F1})$ we get
$$ \k(\Delta) = \sum_{S\in\Omega}\tau(S) = \sum_{S\in\Lambda, S\ne\emptyset}\tau(S) =-(-1)^{d-1} \sum_{S\in\Lambda, S\ne\emptyset}\k(\lk_{\Lambda}(S)).$$
On the other hand, if $S\in\Lambda$ and $S\ne A$, then $\lk_{\Lambda}(S) =\left<A\setminus S\right>$, which implies $\k(\lk_{\Lambda}(S))=0$. Together with the equality above, this fact yields
$$\k(\Delta) = -(-1)^{d-1} \k(\lk_{\Lambda}(A))=-(-1)^{d-1} \k(\{\emptyset\}) = -(-1)^{d-1}(-1) = (-1)^{d-1},$$
as required.
\end{proof}

In what follows we use the exterior product to describe the homology groups of simplicial complexes. Assume that $\Delta$ is a simplicial complex over the vertex set $[n]$. Then the exterior products $\e_F = e_{j_0} \wedge \cdots \wedge e_{j_i}$, where $F = \{j_0,\ldots,j_i\}\in \Delta$ with $j_0 < \cdots < j_i$, form a basis of $\mathcal {\widetilde C}_i(\Delta;k)$.  The differentials $\partial_i: \mathcal {\widetilde C}_i(\Delta;k) \ \longrightarrow  \mathcal {\widetilde C}_{i-1}(\Delta;k)$ of the reduced chain complex $\mathcal {\widetilde C}_{\bullet}(\Delta;k)$ are given by
$$ \partial_i(\ e_{j_0}\wedge \cdots\wedge e_{j_i}) = \sum_{s=0}^i (-1)^se_{j_0} \wedge \cdots \wedge \widehat{e}_{j_s} \wedge \cdots \wedge e_{j_i},$$
and the $i$th homology group of $\Delta$ is $\h_i(\Delta;k)=\ker(\partial_{i})/ \im(\partial_{i+1})$. For simplicity, if $\omega\in \mathcal {\widetilde C}_i(\Delta;k)$, we write $\partial \omega$ instead of $\partial_i \omega$.  With this notation we have
\begin{equation}\label{DF}\partial  (\omega \wedge \nu) =\partial  \omega \wedge \nu + (-1)^{i+1} \omega \wedge \partial  \nu \ \text{ for all } \omega \in
\mathcal {\widetilde C}_i(\Delta;k) \text{ and } \nu \in \mathcal {\widetilde C}_j(\Delta;k).\end{equation}	

\smallskip

In the sequel we also need a concrete representation for elements of $\mathcal {\widetilde C}_i(\Delta;k)$. Let $\omega \in \mathcal {\widetilde C}_i(\Delta;k)$. Then, $\omega = \sum_{F\in\mathcal F_i(\Delta)}\lambda_F \e_F$, where $\mathcal F_i(\Delta) =\{F\in\Delta \mid \dim F=i\}$ and $\lambda_F \in k$ for all $F$. Note that for two faces $V$ and $F$ of $\Delta$ with $V\subseteq F$, by the alternative property of the wedge product we can write $\e_F = \e_V \wedge a_{V,F} \e_{F\setminus V}$ where $a_{V,F} = -1$ or $a_{V,F}=1$. Let $A$ be a face of $\Delta$. Then, we can rewrite $\omega$ as
$$\omega = \sum_{V\subseteq A} \left(\sum_{F\in \mathcal F_i(\Delta), \ F \cap A = V} \lambda_F \e_{F}\right) = \sum_{V\subseteq A} \e_V\wedge \left(\sum_{F\in \mathcal F_i(\Delta), \ F \cap A = V} \lambda_F a_{V,F} \e_{F\setminus V}\right).
$$
For each $V\subseteq A$, let $$\omega_V := \sum_{F\in \mathcal F_i(\Delta), \ F \cap A = V} \lambda_F a_{V,F} \e_{F\setminus V}$$ so that
$\omega_V\in \mathcal {\widetilde C}_{i-|V|}(\lk_{\Delta}(V)\setminus A;k)$. Then,
\begin{equation}\label{representation}\omega =  \sum_{V\subseteq A} \e_V\wedge \omega_V.
\end{equation}

\smallskip

We are now ready to prove the main result of this section.

\begin{prop}\label{W2} If $G$ is a triangle-free graph without isolated vertices, then $G$ is Gorenstein if and only if $G$ is in $W_2$.
\end{prop}
\begin{proof} If $G$ is Gorenstein, then $G$ is in $W_2$ by Lemma $\ref{GW}$.

Conversely, we prove that $G$ is Gorenstein when $G\in W_2$ by induction on $\alpha(G)$. If $\alpha(G)=1$, then $G$ is a complete triangle-free graph. As $G$ is in $W_2$, it has no isolated vertices, so $G$ is an edge. Thus, $G$ is Gorenstein.

Assume that $\alpha(G)\geqslant 2$. Let $\Delta :=\Delta(G)$. By Lemma $\ref{EulerCh}$ we have $\Delta$ is Eulerian. For each vertex $x$ of $G$, by Lemma $\ref{WLC}$ we have $G_x$ is in $W_2$. Since $\alpha(G_x) = \alpha(G)-1$, by the induction hypothesis $G_x$ is Gorenstein. So, $G$ is locally Gorenstein.

Since $\Delta$ is Eulerian, by Lemma $\ref{Stanley}$ we have $\Delta$ is Gorenstein whenever $\Delta$ is Cohen-Macaulay. As $G$ is well-covered and locally Gorenstein, by Reisner's criterion, $G$ is Cohen-Macaulay if
$$\h_i(\Delta;k)= \zv,\ \text{ for all } i <\dim(\Delta).$$
Let $a$ be any vertex of $G$ and let $A:=N_G(a)$. Notice that $A$ is a nonempty independent set of $G$ since $G$ is a triangle-free member of $W_2$.

Given any $i$ with $i< \dim (\Delta)$, in order to prove $\h_i(\Delta;k)=\zv$ we will prove that for every $\omega\in\mathcal {\widetilde C}_i(\Delta;k)$ such that $\partial \omega=0$, there is $\zeta\in \mathcal {\widetilde C}_{i+1}(\Delta;k)$ such that $\omega = \partial \zeta$.
Indeed, by Equation $(\ref{representation})$ we write $\omega$ as
$$\omega = \sum_{V\subseteq A}\e_V\wedge \omega_V$$
where $\omega_V\in \mathcal {\widetilde C}_{i-|V|}(\Delta(G_V\setminus A);k)$.

If $\omega_F = 0$ for any $\emptyset\ne F\subseteq A$, then $\omega \in \mathcal {\widetilde C}_i(\Delta(G\setminus A);k) = \mathcal {\widetilde C}_i(\st_{\Delta}(a);k)$. Note that $\h_i(\st_{\Delta}(a);k) = \zv$ because $\st_{\Delta}(a)$ is a cone over $a$, so $\omega = \partial \zeta$ for some $\zeta\in \mathcal {\widetilde C}_{i+1}(\st_{\Delta}(a);k) \subseteq \mathcal {\widetilde C}_{i+1}(\Delta;k)$.

Assume that $\omega_F\ne 0$ for some $\emptyset\ne F\subseteq A$. Take such an $F$  such that $|F|$ is maximal. By Equation $(\ref{DF}$) we have
$$\partial \omega = \sum_{V\subseteq A}\left(\partial \e_V\wedge \omega_V  + (-1)^{|V|}\e_V\wedge \partial\omega_V \right)= 0,$$
which implies $\partial \omega_F=0$.

We now claim that $\h_{i-|F|}(\Delta(G_F\setminus A);k) =\zv$. Indeed,  if $F=A$, then $\Delta(G_F\setminus A) =\Delta(G_F)$ is Gorenstein with $\dim (\Delta(G_F)) = \dim(\Delta)-|F|$ by Lemma $\ref{ML}$. Note that $i-|F| <\dim(\Delta)-|F| =\dim(\Delta(G_F\setminus A))$, so by Reisner's criterion we have $\h_{i-|F|}(\Delta(G_F\setminus A);k)=\zv$.

If $F$ is a proper subset of $A$, then $G_F\setminus A = G_F\setminus (A\setminus F)$ and $G_F$ is a Gorenstein with $\Delta(G_F)=\core (\Delta(G_F))$ by Lemma $\ref{ML}$. Since $A\setminus F$ is an independent set of $G_F$, by Lemma $\ref{double}$ we have $\h_{i-|F|}(\Delta(G_F\setminus A);k)=\zv$, as claimed.

By our claim we have $\h_{i-|F|}(\Delta(G_F\setminus A);k) =\zv$, hence there is $\eta_F \in \mathcal {\widetilde C}_{i-|F|+1}(\Delta(G_F\setminus A); k)$ such that $\omega_F =\partial( \eta_F)$. Write $F=\{a_1,\ldots,a_s\}$ where $a_1<\cdots<a_s$. Then,
$$\partial \e_F = \sum_{i=1}^s(-1)^{i-1}e_{a_1} \wedge \cdots \wedge \widehat{e}_{a_i} \wedge \cdots \wedge e_{a_s} = \sum_{i=1}^s (-1)^{i-1}\e_{F\setminus\{a_i\}}.$$
Since
$$\partial(\e_F\wedge \eta_F) = \partial \e_F \wedge \eta_F + (-1)^{|F|}\e_F \wedge \partial \eta_F = \partial \e_F \wedge \eta_F + (-1)^{|F|}\e_F \wedge \omega_F,$$
we have
$$\omega -\partial((-1)^{|F|}\e_F\wedge \eta_F) = \sum_{i=1}^s \e_{F\setminus\{a_i\}}  \wedge ((-1)^{|F|+i} \eta_F)+\sum_{V\subseteq A, V\ne F}\e_V\wedge \omega_V.$$
Note that $(-1)^{|F|}\e_F\wedge \eta_F \in \mathcal {\widetilde C}_{i+1}(\Delta;k)$. By repeating this process after finitely many steps, we get an element $\eta\in \mathcal {\widetilde C}_{i+1}(\Delta;k)$ such that
$$\omega - \partial \eta \in \mathcal {\widetilde C}_i(\Delta(G\setminus A);k) = \mathcal {\widetilde C}_i(\st_{\Delta}(a);k).$$
Since $\h_i(\st_{\Delta}(a);k) = \zv$ and $\partial (\omega -\partial \eta) = \partial \omega - \partial ^2 \eta = 0$, there is $\xi\in \mathcal {\widetilde C}_{i+1}(\st_{\Delta}(a);k)\subseteq \mathcal {\widetilde C}_{i+1}(\Delta;k)$ such that
$$\omega - \partial \eta = \partial \xi, \ \text{ i.e., } \  \omega = \partial(\eta +\xi),$$
as required.
\end{proof}

%**************************************************************************************************

\section{Cohen-Macaulayness of the second power of edge ideals}

In this section we complete the proof of the main result by characterizing graphs $G$ such that $I(G)^2$ is Cohen-Macaulay. If $e$ is an edge of a given graph $G$, we can obtain a new graph by deleting $e$ from $G$ but leaving the vertices and the remaining edges intact. The resulting graph is denoted by $G\setminus  e$. An edge $e$ in $G$ is said to be $\alpha$-critical if $\alpha(G \setminus e) > \alpha(G)$. If every edge of $G$ is $\alpha$-critical, $G$ is called $\alpha$-{\it critical}. Staples \cite{Sp1} proved that every triangle-free member of $W_2$ is $\alpha$-critical. We will use this property to give a characterization of triangle-free members of $W_2$. For an edge $ab$ of $G$, let $G_{ab}$ be the induced subgraph $G\setminus (N_G(a) \cup N_G(b))$ of $G$.

\begin{lem} \label{CM1} Let $G$ be a graph. Then we have:
\begin{enumerate}
\item If $G_x$ is well-covered with $\alpha(G_x)=\alpha(G)-1$ for all vertices $x$, then $G$ is well-covered;
\item If $G_{ab}$ is well-covered with $\alpha(G_{ab})=\alpha(G)-1$ for all edges $ab$, then $G$ is well-covered.
\end{enumerate}
\end{lem}
\begin{proof} $(1)$ is obvious.

We now prove $(2)$ by induction on $\alpha(G)$. If $\alpha(G)=1$, then $G$ is a complete graph, and then $\alpha(G_{ab})=0$ for every edge $ab$. Clearly, $G$ is well-covered.

 Assume that $\alpha(G)\geqslant 2$. We may assume that $G$ has no isolated vertices. Let $xy$ be an edge of $G$. Since $G_{xy}$ is an induced subgraph of $G_x$, we have $\alpha(G)-1 = \alpha(G_{xy}) \leqslant \alpha(G_x)$. On the other hand, the inequality $\alpha(G) -1 \geqslant \alpha(G_x)$ is obvious, so $\alpha(G_x)=\alpha(G)-1$. Note that $(G_x)_{ab} = (G_{ab})_x$ for all edges $ab$ of $G_x$. Since $G_{ab}$ is well-covered with $\alpha(G_{ab})=\alpha(G)-1$, by Lemma $\ref{LC}$ we have $(G_x)_{ab}$ is well-covered with $\alpha((G_x)_{ab}) = \alpha((G_{ab})_x) =\alpha(G_{ab})-1 =\alpha(G)-2 =\alpha(G_x)-1$. Therefore, $G_x$ is well-covered by the induction hypothesis. By part $(1)$ we have $G$ is well-covered.
\end{proof}

\begin{lem} \label{CT}Let $G$ be a triangle-free graph without isolated vertices. Then $G$ is in $W_2$ if and only if $G_{ab}$ is well-covered with $\alpha(G_{ab})=\alpha(G)-1$ for all edges $ab$.
\end{lem}
\begin{proof} Assume that $G$ is in $W_2$. We will prove by induction on $\alpha(G)$ that $G_{ab}$ is well-covered with $\alpha(G_{ab})=\alpha(G)-1$ for all edges $ab$. Indeed, if $\alpha(G)=1$, then $G$ is an edge. In this case, $G_{ab}$ is empty, so it is well-covered.

Assume that $\alpha(G) \geqslant 2$. We first prove $\alpha(G_{ab}) =\alpha(G)-1$. Indeed, since $G$ is $\alpha$-critical by \cite[Theorem $3.10$]{Sp1}, we have $\alpha(G\setminus ab) > \alpha(G)$. Let $F$ be a maximal independent set in $G\setminus ab$ such that $|F| =\alpha(G\setminus ab)$. Then, $F\setminus\{a\}$ and $F\setminus\{b\}$ are independent sets in $G$.  Consequently, $|F| - 1 \leqslant \alpha(G)$. This implies $\alpha(G) = |F| -1$, so $\alpha(G\setminus ab) = |F| = \alpha(G)+1$.

Since $\{a,b\}$ is an independent set $G\setminus ab$, we have $G_{ab} = (G \setminus ab)_{\{a,b\}}$ and $F\setminus\{a,b\}$ is an independent set in $G_{ab}$. Thus, $\alpha(G_{ab}) \geqslant |F| - 2 =\alpha(G)-1$. On the other hand, since $G_{ab}$ is an induced subgraph of $G_a$, we have $\alpha(G_{ab}) \leqslant \alpha(G_a) = \alpha(G)-1$. Therefore, $\alpha(G_{ab}) =\alpha(G)-1$, as claimed.

Next, we prove that $G_{ab}$ is well-covered.  For any vertex $x$ of $G_{ab}$, by Lemmas $\ref{LC}$ and $\ref{WLC}$ we have $G_x$ is in $W_2$ and $\alpha(G_x) = \alpha(G)-1$. Note that $(G_{ab})_x = (G_x)_{ab}$, so by the induction hypothesis we have $(G_{ab})_x$ is well-covered with
$$\alpha((G_{ab})_x) = \alpha((G_x)_{ab})=\alpha(G_x)-1=\alpha(G)-2 = \alpha(G_{ab})-1.$$
Hence, $G_{ab}$ is well-covered according to Lemma $\ref{CM1}$.

Conversely, assume that $G_{ab}$ is well-covered and $\alpha(G_{ab})=\alpha(G)-1$ for all edges $ab$. We now prove by induction on $\alpha(G)$ that $G$ is in $W_2$. Note that $G$ is well-covered by Lemma $\ref{CM1}$. Moreover, $\alpha(G\setminus x)=\alpha(G)$ for all vertices $x$ since $G$ is well-covered without isolated vertices. Therefore, it remains to prove that $G\setminus x$ is well-covered. Now, if $\alpha(G)=1$, then $G$ is an edge. In this case, $G\setminus x$ is one vertex, so it is well-covered.

Assume that $\alpha(G)\geqslant 2$.  We first claim that $G_y$ has no isolated vertices for any vertex $y$ of $G$. Assume on contrary that $G_y$ has an isolated vertex $z$. Let $Y:=N_G(y)$ and $Z:=N_G(z)$, so that $Z\subseteq Y$. Note that $Y\ne\emptyset$ and $Z\ne\emptyset$ because $G$ has no isolated vertices. Since $\{y,z\}$ is an independent set of $G$ and $G$ is well-covered, $\alpha(G_{\{y,z\}}) = \alpha(G) - 2$ by Lemma $\ref{LC}$. On the other hand, for any $w$ in $Z$, we have $G_{yw}$ is an induced subgraph of $G_{\{y,z\}}$, hence $\alpha(G_{yw}) \leqslant \alpha(G_{\{y,z\}}) =\alpha(G)-2 < \alpha(G)-1$, a contradiction. Thus, $G_y$ has no isolated vertices, as claimed.

Now we return to prove that $G\setminus x$ is well-covered. It suffices to show that $(G\setminus x)_y$ is well-covered with $\alpha((G \setminus x)_y)=\alpha(G\setminus x)-1$ by Lemma $\ref{CM1}$. Indeed, if $y$ and $x$ are adjacent in $G$, then $(G\setminus x)_y = G_y$, and then $(G\setminus x)_y$ is well-covered with $\alpha((G \setminus x)_y)=\alpha(G_y)=\alpha(G)-1=\alpha(G\setminus x)-1$. If $y$ and $x$ are not adjacent in $G$, then $x$ is a vertex of $G_y$ and  $(G\setminus x)_y = (G_y)\setminus x$. For every edge $ab$ of $G_y$ we have $(G_y)_{ab} = (G_{ab})_y$, so $(G_y)_{ab}$ is well-covered and $\alpha((G_y)_{ab}) =\alpha((G_{ab})_y) =\alpha(G_{ab})-1 = \alpha(G)-1 -1 =\alpha(G_y)-1$. Since $G_y$ has no isolated vertices as the claim above and $\alpha(G_y)=\alpha(G)-1$, by the induction hypothesis we have $G_y$ is in $W_2$. In particular, $(G\setminus x)_y= (G_y)\setminus x$ is well-covered with $\alpha((G\setminus x)_y) = \alpha((G_y)\setminus x) = \alpha(G_y)=\alpha(G)-1=\alpha(G\setminus x)-1$, as required.
\end{proof}

\begin{lem} \label{main2} If $G$ is a triangle-free Gorenstein graph, then $I(G)^2$ is Cohen-Macaulay. \end{lem}

\begin{proof}  We may assume that $G$ has no isolated vertices. By Proposition $\ref{W2}$ we have $G$ is in $W_2$. As $G$ is Gorenstein, by \cite[Corollary $2.3$]{HMT1} it suffices to show that $G_{ab}$ is Cohen-Macaulay with $\alpha(G_{ab})=\alpha(G)-1$ for every edge $ab$ of $G$. Let $ab$ be an edge of $G$. By Lemma $\ref{CT}$ we have $\alpha(G_{ab}) = \alpha(G)-1$. So it remains to prove that $G_{ab}$ is Cohen-Macaulay. Let $A:=N_G(a)\setminus \{b\}$, so that $A$ is a face of $\Delta(G_b)$ since $G$ is triangle-free. Note that $G_{ab} =G_b \setminus A$, so $\Delta(G_{ab}) = \Delta(G_b\setminus A) = \Delta(G_b)\setminus A$. Since $G_b$ is Gorenstein, by Corollary $\ref{C1}$ we have $\Delta(G_b)\setminus A$ is Cohen-Macaulay. It follows that $G_{ab}$ is Cohen-Macaulay, as required.
\end{proof}

Now we are ready to prove the main result of this paper.

\begin{thm} \label{mth}
\it $\ $ Let $G$ be a simple graph without isolated vertices. Then, the following conditions are equivalent:
\begin{enumerate}
\item $G$ is a triangle-free Gorenstein graph;
\item $G$ is a triangle-free member of $W_2$;
\item $I(G)^2$ is Cohen-Macaulay.
\end{enumerate}
\end{thm}
\begin{proof} $(1)\Longleftrightarrow (2)$ follows from Proposition $\ref{W2}$.  $(1)\Longrightarrow (3)$ follows from Lemma $\ref{main2}$. $(3)\Longrightarrow (2)$ follows from \cite[Corollary 2.3]{HMT1} and Lemma $\ref{CT}$ since every Cohen-Macaulay graph is well-covered by \cite[Lemma $6.1.21$]{Vi2}.
\end{proof}

\begin{rem} Let $G$ be a connected graph in $W_2$. If $G$ is not $K_2$ or $C_5$, then $\girth(G) \leqslant 4$ by \cite[Theorem $6$]{Pi}. Therefore, all connected Gorenstein graphs of girth at least $5$ are only $K_1$, $K_2$ and $C_5$;  so that the structure of  triangle-free Gorenstein graphs is non trivial only for the ones of girth $4$. Using the classification of the planar graphs of girth $4$ in $W_2$ due to Pinter \cite{Pi} we obtain the family of the connected planar Gorenstein graphs of girth $4$ is $\{G_n\}_{n\geqslant 3}$ where $G_n$ is the graph on the vertex set $\{1,\ldots,3n-1\}$ with the edge ideal $I(G_n)$ being
$$\{x_1x_2, \{x_{3k-1}x_{3k}, x_{3k}x_{3k+1}, x_{3k+1}x_{3k+2}, x_{3k+2}x_{3k-2}\}_{k=1,\ldots,n-1}, \{x_{3l-3}x_{3l}\}_{l=2,\ldots,n-1}\}.$$
(see also  \cite{HMT2} for another proof)
\end{rem}
\begin{figure}[ht]
\begin{center}
 \includegraphics[scale=0.8]{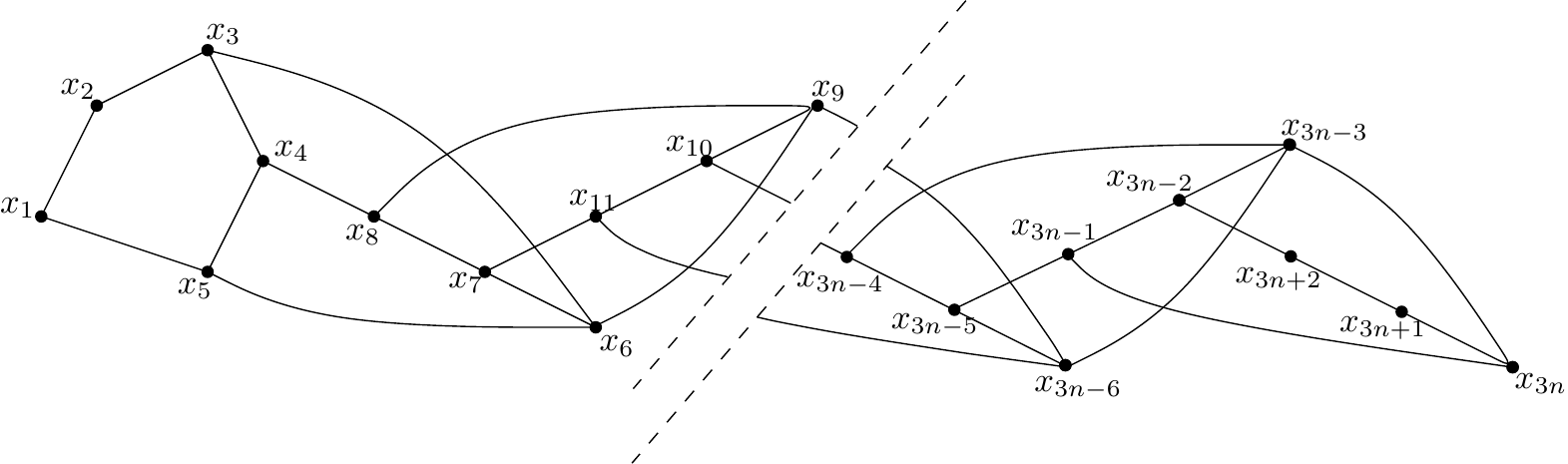}\\
  \caption[Figure 1.]{The graph $G_n$.}\label{Fig:1}
\end{center}
\end{figure}

\subsection*{Acknowledgment} We would like to thank Professors L. T. Hoa and  N. V. Trung for helpful comments. Part of this work was done while we were at the Vietnam Institute of Advanced Studies in Mathematics (VIASM) in Hanoi, Vietnam. We would like to thank VIASM for its hospitality. We would also like to thank anonymous referees for many
helpful comments.

\end{document}